\documentclass{amsart}

\title[Infinite Wordle and the mastermind numbers]{Infinite Wordle and the mastermind numbers}

\author{Joel David Hamkins}
\address[Joel David Hamkins]
{O'Hara Professor of Logic, University of Notre Dame, 100 Malloy Hall, Notre Dame, IN 46556 USA
\textup{and} Associate Faculty Member, Professor of Logic, University of Oxford, United Kingdom
}
\email{jdhamkins@nd.edu}
\urladdr{http://jdh.hamkins.org}

\thanks{I am grateful to Corey Switzer (Vienna) for an extremely helpful discussion about the eventually different number. And thanks also to Toby Ord (Oxford) for a helpful exchange. Commentary can be made about this article on my blog at \href{http://jdh.hamkins.org/infinite-wordle-mastermind}{http://jdh.hamkins.org/infinite-wordle-mastermind}.}
\subjclass{03E17, 91A46}

\usepackage{latexsym,amsfonts,amsmath,amssymb,mathrsfs}
\usepackage{bbm}
\usepackage[hidelinks]{hyperref}
\usepackage[cmyk,svgnames,dvipsnames]{xcolor}
\usepackage{tikz}
\usetikzlibrary{arrows,arrows.meta,petri,topaths,positioning,shapes,shapes.misc,patterns,calc,decorations.pathreplacing,hobby}
\pgfdeclarelayer{foreground}
\pgfdeclarelayer{background}
\pgfdeclarelayer{boardarrows}
\pgfdeclarelayer{boardgrid}
\pgfdeclarelayer{boardshades}
\pgfsetlayers{boardshades,boardgrid,boardarrows,background,main,foreground}
\usepackage{wrapfig} 
\usepackage{caption} 
\DeclareCaptionStyle{compact}{font=footnotesize,format=plain,name={Fig},labelsep=period,skip=1.5ex,margin=0pt,oneside,justification=centering}
\DeclareCaptionStyle{leftside}{style=compact,margin={0pt,9pt}}
\DeclareCaptionStyle{rightside}{style=compact,margin={9pt,0pt}}
\usepackage{float}
\usepackage[utf8]{inputenc}
\RequirePackage{doi}

\usepackage{enumitem}

\newtheorem{theorem}{Theorem}
\newtheorem*{theorem*}{Theorem}

\newtheorem*{maintheorem*}{Main Theorem}
\newtheorem*{maintheorems*}{Main Theorems}

\newtheorem*{corollary*}{Corollary}
\newtheorem*{corollaries*}{Corollaries}

\newtheorem{observation}[theorem]{Observation}

\theoremstyle{definition}

\newtheorem*{definition*}{Definition}

\newtheorem{question}[theorem]{Question}
\newtheorem*{question*}{Question}

\newtheorem*{questions*}{Questions}
\newtheorem*{mainquestion*}{Main Question} 
\newtheorem*{openquestion*}{Open Question} 
\theoremstyle{remark}

\newcommand{\QED}{\end{proof}}

\def\proclaim[#1]{{\bf #1}}
\def\BF#1.{{\bf #1.}}

\def\says#1:#2\par{\item[#1] #2\par}

\newcommand{\C}{{\mathbb C}}

\renewcommand{\P}{{\mathbb P}}

\newcommand{\R}{{\mathbb R}}

\newcommand{\continuum}{\mathfrak{c}}

\makeatletter
\newcommand{\dotminus}{\mathbin{\text{\@dotminus}}}
\newcommand{\@dotminus}{%
  \ooalign{\hidewidth\raise1ex\hbox{.}\hidewidth\cr$\m@th-$\cr}%
}
\makeatother

\newcommand{\of}{\subseteq}

\newcommand{\set}[1]{\{\,{#1}\,\}}

\newcommand{\ran}{\mathop{\rm ran}}

\newcommand{\restrict}{\upharpoonright} 

\newcommand\dbrace{\hskip-1.5em\raise3pt\hbox{\rotatebox[origin=c]{-35}{$\left.\strut^{\phantom{|}}\right\}$}}}

\newcommand\UParroW{{\setbox0\hbox{$\Uparrow$}\rlap{\hbox to \wd0{\hss$\mid$\hss}}\box0}}

\renewcommand{\setminus}{\raise.3ex\hbox{\rotatebox{-20}{$-$}}} 

\newcommand{\Union}{\bigcup}

\newcommand{\smalllt}{\mathrel{\mathchoice{\raise2pt\hbox{$\scriptstyle<$}}{\raise1pt\hbox{$\scriptstyle<$}}{\raise0pt\hbox{$\scriptscriptstyle<$}}{\scriptscriptstyle<}}}
\newcommand{\smallleq}{\mathrel{\mathchoice{\raise2pt\hbox{$\scriptstyle\leq$}}{\raise1pt\hbox{$\scriptstyle\leq$}}{\raise1pt\hbox{$\scriptscriptstyle\leq$}}{\scriptscriptstyle\leq}}}

   \def\DHLhksqrt#1#2{%
   \setbox0=\hbox{$#1\sqrt{#2\,}$}\dimen0=\ht0
   \advance\dimen0-0.2\ht0
   \setbox2=\hbox{\vrule height\ht0 depth -\dimen0}%
   {\box0\lower0.4pt\box2}}

\def\[#1]{\mathopen{\lbrack\!\lbrack}#1\mathclose{\rbrack\!\rbrack}}
\newbox\gnBoxA
\newbox\gnBoxB
\newdimen\gnCornerHgt
\setbox\gnBoxA=\hbox{\tiny$\ulcorner$}
\global\gnCornerHgt=\ht\gnBoxA
\newdimen\gnArgHgt
\def\gcode #1{%
\setbox\gnBoxA=\hbox{$#1$}%
\setbox\gnBoxB=\hbox{$\bar #1$}%
\gnArgHgt=\ht\gnBoxB%
\ifnum     \gnArgHgt<\gnCornerHgt \gnArgHgt=0pt%
\else \advance \gnArgHgt by -\gnCornerHgt%
\fi \raise\gnArgHgt\hbox{\tiny$\ulcorner$} \box\gnBoxA %
\raise\gnArgHgt\hbox{\tiny$\urcorner$}}
\newcommand{\UnderTilde}[1]{{\setbox1=\hbox{$#1$}\baselineskip=0pt\vtop{\hbox{$#1$}\hbox to\wd1{\hfil$\sim$\hfil}}}{}}
\newcommand{\Undertilde}[1]{{\setbox1=\hbox{$#1$}\baselineskip=0pt\vtop{\hbox{$#1$}\hbox to\wd1{\hfil$\scriptstyle\sim$\hfil}}}{}}
\newcommand{\undertilde}[1]{{\setbox1=\hbox{$#1$}\baselineskip=0pt\vtop{\hbox{$#1$}\hbox to\wd1{\hfil$\scriptscriptstyle\sim$\hfil}}}{}}
\newcommand{\UnderdTilde}[1]{{\setbox1=\hbox{$#1$}\baselineskip=0pt\vtop{\hbox{$#1$}\hbox to\wd1{\hfil$\approx$\hfil}}}{}}
\newcommand{\Underdtilde}[1]{{\setbox1=\hbox{$#1$}\baselineskip=0pt\vtop{\hbox{$#1$}\hbox to\wd1{\hfil\scriptsize$\approx$\hfil}}}{}}

\def\<#1>{\left\langle#1\right\rangle}

\newcommand{\ZFC}{{\rm ZFC}}

\newcommand{\CH}{{\rm CH}}

\newcommand{\cell}[1]{\boxit{\hbox to 17pt{\strut\hfil$#1$\hfil}}}
\newcommand{\head}[2]{\lower2pt\vbox{\hbox{\strut\footnotesize\it\hskip3pt#2}\boxit{\cell#1}}}
\newcommand{\boxit}[1]{\setbox4=\hbox{\kern2pt#1\kern2pt}\hbox{\vrule\vbox{\hrule\kern2pt\box4\kern2pt\hrule}\vrule}}
\newcommand{\Col}[3]{\hbox{\vbox{\baselineskip=0pt\parskip=0pt\cell#1\cell#2\cell#3}}}
\newcommand{\tapenames}{\raise 5pt\vbox to .7in{\hbox to .8in{\it\hfill input: \strut}\vfill\hbox to
.8in{\it\hfill scratch: \strut}\vfill\hbox to .8in{\it\hfill output: \strut}}}
\newcommand{\Head}[4]{\lower2pt\vbox{\hbox to25pt{\strut\footnotesize\it\hfill#4\hfill}\boxit{\Col#1#2#3}}}
\newcommand{\Dots}{\raise 5pt\vbox to .7in{\hbox{\ $\cdots$\strut}\vfill\hbox{\ $\cdots$\strut}\vfill\hbox{\
$\cdots$\strut}}}
\usepackage[backend=bibtex,style=alphabetic,maxbibnames=15,maxcitenames=6,dateabbrev=false]{biblatex}
\bibliography{HamkinsBiblio,MathBiblio,PhilBiblio,WebPosts}
\renewcommand{\UrlFont}{} 
\renewbibmacro{in:}{\ifentrytype{article}{}{\printtext{\bibstring{in}\intitlepunct}}} 
\DeclareFieldFormat{url}{{\UrlFont\url{#1}}} 
\DeclareFieldFormat{urldate}{
  (version \thefield{urlday}\addspace%
  \mkbibmonth{\thefield{urlmonth}}\addspace%
  \thefield{urlyear}\isdot)}
\DeclareFieldFormat{eprint:arxiv}{
\ifhyperref
    {\href{http://arxiv.org/abs/#1}{%
    arXiv\addcolon\nolinkurl{#1}}\iffieldundef{eprintclass}{}{\UrlFont{\mkbibbrackets{\thefield{eprintclass}}}}}
    {arXiv\addcolon\nolinkurl{#1}\iffieldundef{eprintclass}{}{\UrlFont{\mkbibbrackets{\thefield{eprintclass}}}}}}
\newcommand\Mastermind[1]{%
 \foreach \c in {#1} { node[peg=\c] {} ++(4,0) }}
\newcommand\Masterresult[4]{%
 +(45:1) node[rpeg=#2] {}
 +(135:1) node[rpeg=#1] {}
 +(225:1) node[rpeg=#3] {}
 +(-45:1) node[rpeg=#4] {}
}
\newcommand\Wordle[1]{\foreach \w/\c [count=\x from 0] in {#1} { +(\x,0) node[fill=\c] {\textbf{\w}}}}
\newcommand\Nerdle[1]{\foreach \w/\c [count=\x from 0] in {#1} {+(\x,0) node[fill=\c] {\textbf{\w}} }
}
\newcommand{\inj}{{\overset{\hfil\scriptstyle\omega}{\scriptscriptstyle\hookrightarrow}}}
\newcommand\mm{\mathbbm{m}}
\newcommand\mms{\mm_{\scriptscriptstyle=,\neq}}
\newcommand\mmeq{\mm_{\scriptscriptstyle=}}
\newcommand\mmneq{\mm_{\scriptscriptstyle\neq}}
\newcommand\mmstr{\widehat{\mm}}
\newcommand\mmstrs{\mmstr_{\scriptscriptstyle=,\neq}}
\newcommand\mmstreq{\mmstr_{\scriptscriptstyle=}}
\newcommand\mmstrneq{\mmstr_{\scriptscriptstyle\neq}}
\DeclareMathOperator{\cov}{\textup{cov}}

\begin{document}

\begin{abstract}
I consider the natural infinitary variations of the games Wordle and Mastermind, as well as their game-theoretic variations Absurdle and Madstermind, considering these games with infinitely long words and infinite color sequences and allowing transfinite game play. For each game, a secret codeword is hidden, which the codebreaker attempts to discover by making a series of guesses and receiving feedback as to their accuracy. In Wordle with words of any size from a finite alphabet of $n$ letters, including infinite words or even uncountable words, the codebreaker can nevertheless always win in $n$ steps. Meanwhile, the \emph{mastermind number} $\mm$, defined as the smallest winning set of guesses in infinite Mastermind for sequences of length $\omega$ over a countable set of colors without duplication, is uncountable, but the exact value turns out to be independent of \ZFC, for it is provably equal to the eventually different number $\frak{d}({\neq^*})$, which is the same as the covering number of the meager ideal $\cov(\mathcal{M})$. I thus place all the various mastermind numbers, defined for the natural variations of the game, into the hierarchy of cardinal characteristics of the continuum.
\end{abstract}

\maketitle

\section{Infinite Wordle}

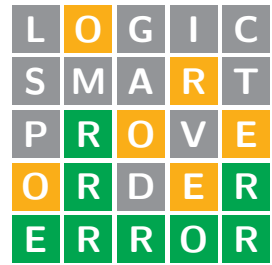
\begin{wrapfigure}{r}{.3\textwidth}\vskip-3ex\hfill
\begin{tikzpicture}[scale=.7,transform shape,every node/.style={white,font=\sffamily,text width=4.5mm,minimum height=4.5mm,inner sep=0pt,outer sep=0pt,very thick,line join=bevel,align=center,scale=2}]
\draw (0,5) \Wordle{L/gray,O/Orange,G/gray,I/gray,C/gray};
\draw (0,4) \Wordle{S/gray,M/gray,A/gray,R/Orange,T/gray};
\draw (0,3) \Wordle{P/gray,R/green,O/Orange,V/gray,E/Orange};
\draw (0,2) \Wordle{O/Orange,R/green,D/gray,E/Orange,R/green};
\draw (0,1) \Wordle{E/green,R/green,R/green,O/green,R/green};
\end{tikzpicture}
\captionsetup{style=rightside}
\caption{A play of Wordle}\label{Figure.Wordle}
\end{wrapfigure}
In the popular game of Wordle, going viral daily online, the codebreaking player aims to discover a secret codeword by making a series of guesses at it, receiving feedback at each stage concerning the accuracy of the letters in those guesses. Specifically, after each guessword is placed, as in figure~\ref{Figure.Wordle}, correct letters are highlighted in green and wrongly placed letters in yellow (limited by the total number of occurrences of that letter), while the remaining incorrect letters appear in gray. The standard game is played with five-letter words from a common English dictionary, but of course there are natural generalizations to longer words and other languages.

\enlargethispage{15pt}
I should like to consider the generalization of Wordle to dictionaries having much longer (perhaps imaginary) words and more generous alphabets of symbols.
$$
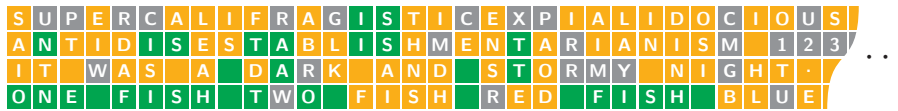
\begin{tikzpicture}[scale=.35,transform shape,every node/.style={white,font=\sffamily,text width=4.5mm,minimum height=4.5mm,inner sep=0pt,outer sep=0pt,very thick,line join=bevel,align=center,scale=2}]
\draw (33,3.5) node[black,scale=2] {$\cdots$};
\clip (-1,1.5) -- (-1,5.5) -- (31.9,5.5) to[out=-90,in=90,looseness=1.6] (30.7,1.5) -- cycle;
\draw (0,5) \Wordle{S/Orange,U/gray,P/gray,E/Orange,R/gray,C/gray,A/Orange,L/Orange,I/Orange,F/Orange,R/gray,A/Orange,G/gray,I/green,S/green,T/Orange,I/Orange,C/gray, E/Orange,X/gray,P/gray,I/Orange,A/Orange,L/Orange,I/Orange,D/Orange,O/Orange,C/gray,I/Orange,O/Orange,U/gray,S/Orange,/Orange};
\draw (0,4) \Wordle{A/Orange,N/green,T/Orange,I/Orange,D/Orange,I/green,S/green,E/Orange,S/Orange,T/green,A/green,B/Orange,L/Orange,I/green,S/green,H/Orange, M/gray,E/Orange,N/Orange,T/green,A/Orange,R/gray,I/Orange,A/Orange,N/Orange,I/Orange,S/Orange,M/gray, /Orange,1/gray,2/gray,3/gray,/gray};
\draw (0,3) \Wordle{I/Orange,T/Orange, /Orange,W/gray,A/Orange,S/Orange, /Orange,A/Orange, /green,D/Orange,A/green,R/gray,K/Orange, /Orange,A/Orange,N/Orange,D/Orange, /green,S/Orange,T/green,O/Orange,R/gray,M/gray,Y/gray, /Orange,N/Orange,I/Orange,G/gray,H/Orange,T/Orange,./Orange,/Orange,/Orange};
\draw (0,2) \Wordle{O/green,N/green,E/green, /green,F/green,I/green,S/green,H/green, /green,T/green,W/gray,O/green, /green,F/Orange,I/Orange,S/Orange,H/Orange, /green,R/gray,E/Orange,D/Orange, /green,F/green,I/green,S/green,H/green, /green,B/Orange,L/Orange,U/gray,E/Orange,/Orange};
\end{tikzpicture}$$
Ultimately I shall place my main focus on \emph{infinite Wordle}, in which we play infinitely long words, that is, infinite sequences of letters chosen from a fixed countable alphabet $\Sigma$, using a fixed nonempty dictionary $\Delta\of{}^\omega\Sigma$ specifying exactly the set of allowed codewords and guesswords. One can similarly consider Wordle with uncountable words or indeed words indexed by any fixed index set $I$ with a corresponding dictionary $\Delta\of\Sigma^I$.

For any word length, whether finite or infinite, the game begins with a hidden codeword $w$ chosen from the dictionary, and the player then makes a series of candidate guesswords $s_0$, $s_1$, $s_2$, and so on, getting the Wordle feedback information for each guessword as to which letters are correctly placed, which letters are correct but misplaced, and which letters of the guessword cannot be made correct by rearrangement of the incorrect letters. To make things definite for the case where the guessword has repeated letters, we first color the correctly placed letters green, and then, amongst the misplaced letters, the first $k$ of them should be labeled yellow if there are indeed $k$ additional such letters appearing elsewhere in the codeword; all other letters remain gray. In this way, the player can learn information about how many instances of a misplaced letter might occur in the codeword. The codeword and all guesswords must be chosen from the fixed dictionary $\Delta$ of allowed words. In the infinite case, it may be natural to consider transfinite play, with guesses $s_\alpha$ at ordinal stage $\alpha$, not necessarily finite.

Wordle is not a game of perfect information, precisely because the codeword is hidden---if the codeword were known openly, the player could of course win on the first guess. What we mean by a Wordle strategy, therefore, is a guessing procedure that tells the codebreaker what word to guess next as a function of the feedback information received for the previously placed guesswords; such a strategy is winning, if for any codeword, the strategy leads at some stage to that word being placed as the guessword.

To my way of thinking, the main initial question for these long-form Wordle variations is---how long will it take the codebreaker to win?

\begin{question}
 In long-form Wordle with a dictionary of extremely long words, perhaps millions of letters in each word, will it also generally take the codebreaker a long time to win? In infinite Wordle, would it be unreasonable for the codebreaker to expect to win always at some finite stage? If words have some fixed uncountable size $\kappa$, can we express the length of the optimal winning times in terms of $\kappa$?
\end{question}

Let me draw out the answer, which I had initially found suprising. The first thing to notice is that in some trivial cases, for example, with the complete dictionary containing all possible sequences of letters as allowed words, the codebreaker can win fairly quickly, even when the words are infinite. Namely, by placing the constant-letter words in an alphabet of size $n$, beginning for example with AAAAA\dots A and then BBBBB\dots B and so on, the codebreaker can learn the true value of every cell in $n$ steps. Indeed, she need only use $n-1$ constant-letter words, since any letter place not yet indicated in green explicitly will be known to be the only remaining letter, and so the correct guess can be made at stage $n$. The conclusion is that with the complete dictionary on an alphabet of size $n$, the codebreaker can win within $n$ steps, even with extremely long or infinite words. This observation pours cold water on the expectation that extremely long words necessarily require a long play to win. But perhaps the dictionary does not allow these constant-letter words; could a devious dictionary prevent this easy win and somehow require very long play?

Consider the variation of Wordle known as \emph{Nerdle}, in which the ``words'' are exactly the valid mathematical equations that can be expressed using the symbols $1234567890{+}{-}{\times}{/}{=}$. This is actually an instance of Wordle in the generality I defined it, because we have simply specified a particular dictionary, the valid equations, over a particular alphabet of symbols, the symbols of basic arithmetic. The standard Nerdle rules call for equations of length $8$ (and also stipulate that there should be no operations appearing to the right of the equality sign), but let us consider equations of longer arbitrary fixed length, as in the sample play here:

$$\colorlet{g}{green}
\colorlet{p}{Orange}
\colorlet{b}{gray}
\begin{tikzpicture}[scale=.35,transform shape,every node/.style={white,font=\sffamily,text width=4.5mm,minimum height=4.5mm,inner sep=0pt,outer sep=0pt,very thick,line join=bevel,align=center,scale=2}]
\draw
++(0,-1) \Nerdle{7/b,7/b,7/b,7/b,7/b,7/b,7/b,7/b,7/b,7/b,7/b,=/p,7/b,7/b,7/b,7/b,7/b,7/b,7/b,7/b,7/b,7/b,7/b}
++(0,-1) \Nerdle{8/b,8/b,8/b,8/b,8/b,8/b,8/b,8/b,8/b,8/b,+/b,1/p,=/p,9/b,9/b,9/b,9/b,9/b,9/b,9/b,9/b,9/b,9/b}
++(0,-1) \Nerdle{3/b,3/b,3/b,3/b,3/b,3/b,3/b,-/b,2/g,2/p,2/g,2/p,2/g,2/p,2/g,=/p,1/p,1/b,1/b,1/b,1/b,1/b,1/g}
++(0,-1) \Nerdle{2/p,*/b,2/g,*/b,2/g,*/b,2/g,*/b,2/g,*/b,2/g,*/b,2/g,*/b,2/g,*/b,2/g,{/}/p,5/p,1/p,2/g,=/g,1/g}
++(0,-1) \Nerdle{5/g,1/g,2/g,{/}/g,2/g,{/}/g,2/g,{/}/g,2/g,{/}/g,2/g,{/}/g,2/g,{/}/g,2/g,{/}/g,2/g,{/}/g,2/g,{/}/g,2/g,=/g,1/g};
\end{tikzpicture}$$

Indeed, I propose that we consider Nerdle with significantly longer equations, perhaps millions of symbols in each equation or more. Will it take a correspondingly long time to win? The surprising answer is no! Regardless of how long the equations are, the codebreaker can always win within fifteen guesses.

\begin{theorem}\label{Theorem.Nerdle-in-fifteen}
 In the game of Nerdle played with equations of any fixed size, even with equations of millions of symbols or more, the codebreaker can win within fifteen guesses.
\end{theorem}

\begin{proof}
Let the codebreaker play any valid equation as the first guess. Perhaps some of the tiles turn green, while others remain yellow or gray. As a second guess, let the codebreaker place an equation that agrees on all the green tiles, but uses a new symbol on every non-green tile. Why is there a valid equation like that? Because the true codeword equation has this feature. With this second guess, perhaps some additional tiles turn green. For the third guess, therefore, let the codebreaker again place an equation that agrees on all the green tiles and at each non-green tile uses a new symbol not yet seen on that tile. Why is there an equation like that? Because the true codeword equation has that property. Continue in this way at each stage, placing an equation that agrees on the all the green tiles, and uses a new symbol not yet seen at each non-green tile.

The main observation is that this strategy will find the true equation within fifteen stages, since at each stage, the remaining non-green tiles use a new symbol not yet appearing in that position. But there are only fifteen symbols in the alphabet, $1234567890{+}{-}{\times}{/}{=}$, and so this can happen at most fifteen times. So the codebreaker will win with at most fifteen guesses, regardless of the length of the equations.
\end{proof}

Of course the argument has nothing to do with Nerdle as opposed to Wordle. The point is that all that matters is the total number of symbols in the alphabet---the word length is ultimately unimportant, and it is simply not true that much longer words will necessarily require a much larger number of guesses to win. The proof works with Wordle for words of any size, even infinite Wordle, even if the words are uncountable.\goodbreak

\begin{theorem}\label{Theorem.Wordle-finite-alphabet}
 In the game of Wordle using words of any given size (including infinite words, even uncountable words) over a finite alphabet of $n$ letters, and any fixed nonempty dictionary, the codebreaker has a strategy that will win within $n$ steps. Furthermore, there is such a strategy for which successive guesses remain consistent with all previous feedback color information, and also one that uses only the green feedback information.
\end{theorem}

\begin{proof}
Let the codebreaker propose any initial guessword $s_0$ from the dictionary. The green feedback tells her which of these letters is correct. As a second guess, let the codebreaker place a word $s_1$ that agrees in exactly those green cells, and uses a new letter in all the non-green cells (and also conforms with all the yellow information). Why is there such a codeword? Because the true codeword itself is like that. And so on. At each stage, she will have some information about which letters are correct, and which places have not yet been seen to be correct, and she should place a next word that agrees with all the green cells and for each non-green cell uses a new letter not yet seen in that cell. There is such a word, because the true codeword is like that. After $n-1$ steps, she will thereby have complete information about the codeword, since for every place, either she has seen a green letter in that place, or she has seen $n-1$ non-green letters in that place, and so there is only one remaining possibility consistent with that information. So the codebreaker will be able to win by stage $n$. (Note: this same argument was also discovered independently by Lokshtanov and Subercaseaux \cite[lemma~3]{Lokshtanov2022:wordle-is-np-hard}.)
\end{proof}

In the infinite case, the result here applies for example to infinite Wordle with trinary words of length $\omega_1$ using any dictionary $\Delta\of{}^\omega\set{0,1,2}$---any such game can be won within three guesses.

To what extent does the phenomenon extend to Wordle with infinite alphabets? In the general case of an alphabet $\Sigma$ with at least two letters, then there will be uncountably many infinite words in the word space $^\omega\Sigma$, and of course the dictionary $\Delta$ might be likewise uncountable. In the special case of a countable dictionary, however, then the codebreaker can play so as to win at a finite stage of play.

\begin{observation}\label{Observation.Wordle-countable-dictionary}
In infinite Wordle with a dictionary of only countably many words, the codebreaker can play so as always to win at a finite stage of play.
\end{observation}

\begin{proof}
This observation was also made by Toby Ord. The codebreaker can simply fix an enumeration of the dictionary and play the words in that order. Since the true codeword is one of the words in the dictionary, eventually there will come a stage of play at which the guessword is fully correct.
\end{proof}

The countability of the dictionary is not a necessary condition for winning at finite stages, however, in light of the earlier arguments with the full dictionary over a finite alphabet. Let us next consider the case of a countably infinite alphabet, but where every word in the dictionary uses only finitely many letters.

\begin{theorem}
If $\Delta$ is an infinite Wordle dictionary on a countable alphabet $\Sigma$ and every word in the dictionary uses only finitely many letters, then the codebreaker has a strategy that will always win at some finite stage of play.
\end{theorem}

\begin{proof}
There are only countably many finite subalphabets $\Sigma_0\of\Sigma$, and for each such subalphabet, theorem \ref{Theorem.Wordle-finite-alphabet} provides a strategy to win in $|\Sigma_0|$ many steps, if the codeword should happen to come from that subalphabet. So let the codebreaker simply attempt all those strategies one after the other. Since the codeword must arise from one such finite subalphabet, eventually the codebreaker will be making an attempt for the relevant subalphabet, and will therefore win by the conclusion of that series of stages.
\end{proof}\goodbreak

A more general observation is the following:

\begin{theorem}\label{Theorem.Wordle-countable-unions}
If an infinite Wordle dictionary $\Delta$ can be expressed as a countable union $\Delta=\Union_n\Delta_n$ of dictionaries $\Delta_n\of{}^\omega\Sigma$, and for each $\Delta_n$ the codebreaker has a strategy always winning at a finite stage, then the codebreaker has a strategy to win with $\Delta$ always at a finite stage.
\end{theorem}

\begin{proof}
The codebreaker can interleave the strategies. Namely, at stage $2^n(2k+1)$ let the codebreaker place the $k$th guess in the $\Delta_n$ strategy, consulting the feedback from the earlier stages $2^n(2k'+1)$ for $k'<k$. In this way, every strategy will become part of the ultimate play, but the true codeword is in some $\Delta_n$, and so that part of the play will win at some finite stage.
\end{proof}

In the general case of an arbitrary dictionary on an infinite alphabet, it may seem natural to allow the game play to proceed transfinitely, with the codebreaker playing a candidate guess word $s_\alpha$ at ordinal stages $\alpha$. When the argument of theorem \ref{Theorem.Wordle-finite-alphabet} is applied with an infinite alphabet, then at any given countable stage, one might still have some non-green squares, since the strategy of that proof only calls for the next guess to agree on all the green tiles and use a new not-yet-used-in-that-place symbol on the non-green tiles. But at a countable ordinal stage, this won't necessarily yet determine the correct letter for such a place, since there could be several letters left, even if we have already used infinitely many. Nevertheless, when the alphabet $\Sigma$ is countable, I claim, then regardless of the dictionary, the codebreaker can win by $\omega$ at the latest.

\begin{theorem}\label{Theorem.Wordle-win-by-omega}
For any dictionary in infinite Wordle with words of length $\omega$ over a countable alphabet, there is a uniform countable sequence of guesswords that will enable the codebreaker in infinite Wordle to correctly guess the codeword at stage $\omega$ or earlier.
\end{theorem}

\begin{proof}
For each letter position, there are only countably many letters that could appear in that position. For each such letter/position combination, if there is a word in the dictionary with that letter appearing in that position, then let the codebreaker simply select one such word with that feature. This will altogether form a list of countably many guesswords, which the codebreaker can guess in sequence. By stage $\omega$, the codebreaker will therefore know of each position which letter appears in it (and this requires only the green feedback information, not the yellow), and therefore has complete knowledge of the codeword at stage $\omega$.
\end{proof}

We can alternatively mount a modified version of the argument of theorem \ref{Theorem.Wordle-finite-alphabet} here. Namely, the codebreaker could make any first guess $s_0$. Some of the places may now show green. Since the codeword is correct for exactly these places, we can restrict attention to the dictionary words that agree with those green letters, and display different letters in all other places. Furthermore, for the first non-green cell, there will be a word like that whose letter in that cell is earliest in the enumeration of the alphabet. So play such a word as $s_1$. Similarly at any finite stage, the codebreaker should play a word from the dictionary that agrees with all the places currently known to be green, and at the leftmost non-green cell, has the smallest possible letter among all such words. In this way, we shall eventually come to know the letter in that leftmost cell, and so after $\omega$ many steps we shall know all the cells. So we can guess the correct codeword at stage $\omega$. This idea shows that there is a winning strategy with the added attractive feature that successive guesses all conform with the feedback information known at that stage, and furthermore, the guesses made at finite stages converge bit-wise to the true codeword, which is the limit of those guesses.

\begin{question}
  For which dictionaries $\Delta\of{}^\omega\Sigma$ in a countably infinite alphabet $\Sigma$ can the codebreaker have a strategy always to discover the codeword at a finite stage of play?
\end{question}

I am not quite sure of the answer. I have given some sufficient conditions, and let me prove now that in the complete dictionary over an infinite alphabet, the codebreaker cannot force a win at a finite stage of play.

\begin{theorem}\label{Theorem.Wordle-not-finite-stage}
 In infinite Wordle with the complete dictionary $\Delta$ having all words of length $\omega$ using an infinite alphabet $\Sigma$, the codebreaker has no strategy of guesswords that will ensure a win at a finite stage of play.
\end{theorem}

\begin{proof}
Suppose toward contradiction that the codebreaker has a strategy enabling her to find any given codeword in finitely many steps. Let me exhibit a codeword and a play of the game using that codeword, which accords with the strategy, but does not find it at any finite stage of play. To do this, I shall embark on a play of the game without fixing a specific codeword in advance, but rather I shall make increasing promises about it as the play unfolds. The codeword I ultimately create will use its letters each exactly once, and these letters will include every letter appearing in any of the guesswords during the play of the game, while remaining consistent with all the answers provided during the play.

The strategy begins with a guessword $s_0$. Let me mark the first appearance of every letter yellow, except if a letter appears cofinitely often, that is, on all but a finite number of positions; in this case, for that letter, I shall mark one of those occurrences green. At some finite stage of play, there have been finitely many guesswords $s_0,s_1,\ldots,s_n$, with the new guessword $s_n$ just placed. Assume inductively that I have made only finitely many green promises so far about definite locations of letters in the codeword, together with the ongoing yellow promises that every letter will appear exactly once. If it happens at stage $n$ that there is a letter that has now for the first time been mentioned in the guesses up to this stage at all but finitely many places, then this letter must appear infinitely often in $s_n$ at places in which it did not appear earlier, and I may select one of those places to mark green. Notice that there can be at most $n+1$ many letters with this property, since each position has had at most this many letters mentioned so far, and so I shall make only finitely many green promises at this stage for this reason. Next, for the leftmost cell about which I have not yet committed myself, I may select a letter not yet appearing in that position and not yet promised for green at another location, and promise to myself that the codeword will have that letter in that position. Finally, for the next letter about which I have not yet determined its place, then there must be infinitely many places for which it has not yet appeared at that position in any of the guesswords so far (since otherwise we handled it above), and I may promise to myself that the codeword will have that letter in that place. All other letters in $s_n$ are marked yellow for the first occurrence, unless I have already made a commitment about that letter in the other dealings.

In this way, the promises I make about the ultimate codeword will remain consistent with the answers I have given, while revealing only finitely many definite bits of the codeword at any stage. Since I made sure for each cell eventually to commit to a particular letter for that cell and for every letter to commit to a particular place for that letter, the promises will converge ultimately to a unique codeword. Thus, the overall play is coherent with respect to that codeword, but the strategy did not succeed at any finite stage in discovering it.
\end{proof}

\section{Infinite Absurdle}

The proof technique used in the proof of theorem \ref{Theorem.Wordle-not-finite-stage} suggests a competitive process between the codebreaking player and a codemaker striving to delay the win as long as possible. And indeed, this is the precisely the pattern of game play of the extremely interesting variation of Wordle known as \emph{Absurdle}, a two-player variation played between the codebreaker and a codemaking player called the absurdist. Namely, the codebreaker plays Wordle as usual, while the absurdist gives the Wordle feedback as the game proceeds, but does not commit to a definite codeword in advance. Rather, the absurdist can in effect change his mind about what the codeword is as play proceeds, provided that the new codeword remains consistent with all previous feedback answers. It may be best to imagine that the absurdist is thinking at each stage of a set of possible codewords, each of them consistent with the answers given thus far, and he is striving to provide answers to new queries in such a way that gives as little information as possible. Play proceeds transfinitely, and the codebreaker wins if the guessword is the only word in the dictionary consistent with the answers given so far or if at some stage of play there is no codeword in the dictionary that is consistent with all answers given thus far.

Precisely because there is no longer a hidden codeword, Absurdle is a game of perfect information, and thus the entire theory of games of perfect information comes into force. The case of infinite Absurdle played with countable words on a countable alphabet and play of length $\omega$ fits well into the set-theoretic analysis of determinacy for games on the real numbers. Let us say that the codebreaker wins at a finite stage by achieving the all-green-tiles feedback answer; but also she wins after infinite play if there is no codeword consistent with all the information that the absurdist had provided. The absurdist wins if play proceeds infinitely without ever giving all-green tiles as feedback at any stage, yet there remains a codeword in the dictionary consistent with all the color feedback information that had been given. So we can view this as a game of length $\omega$ played on the real numbers, as in the context of determinacy questions for games on $\R$.

\begin{question}
 For which dictionaries can the codebreaker win in infinite Absurdle? For which dictionaries can the codebreaker win always at a finite stage of play?
\end{question}

These seem to be subtle questions of set theory, and I do not have a complete answer. Meanwhile, much of our earlier analysis of infinite Wordle carries over to Absurdle.

\begin{theorem}\ \label{Theorem.Absurdle}
  \begin{enumerate}
    \item In infinite Absurdle with any dictionary over a finite alphabet of size $n$, the codebreaker can win in $n$ steps.
    \item In infinite Absurdle with any countable dictionary, the codebreaker has a strategy to ensure winning in the game of length $\omega$. 
    \item In infinite Absurdle with any dictionary over a countable alphabet and transfinite play, there is a strategy for the codebreaker to win at stage $\omega$ or earlier.
    \item In infinite Absurdle with the dictionary of all words on a countably infinite alphabet, there is no strategy ensuring that the codebreaker will win at a finite stage of play---the absurdist can win the game of length $\omega$.
    \item The dictionaries for which the codebreaker has a winning strategy to win at a finite stage in infinite Absurdle is closed under countable unions.
  \end{enumerate}
\end{theorem}

\begin{proof}
These claims are proved just as in theorem \ref{Theorem.Wordle-finite-alphabet}, observation \ref{Observation.Wordle-countable-dictionary}, theorem \ref{Theorem.Wordle-win-by-omega}, theorem \ref{Theorem.Wordle-not-finite-stage}, and theorem \ref{Theorem.Wordle-countable-unions}. The point is that from the codebreaker's perspective, every game of infinite Wordle might as well be a game of infinite Absurdle, since every play of Absurdle is a legal play of Wordle. Indeed, from the codebreaker's perspective, every play of Absurdle is a worst-case play of Wordle.
\end{proof}

\section{Infinite Mastermind}

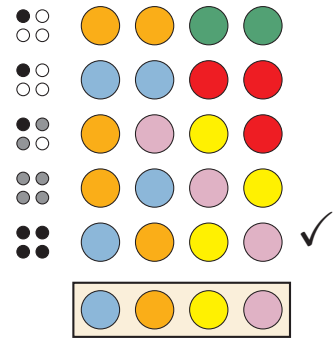
\begin{wrapfigure}{r}{.32\textwidth}\hfill
\begin{tikzpicture}[scale=.18,transform shape,peg/.style={circle,inner sep=1cm,draw,fill=#1},peg/.default={white},rpeg/.style={circle,inner sep=.33cm,draw,thin,fill=#1},rpeg/.default={white}]
\useasboundingbox (0,-2) rectangle (19,23);
\draw [thick,fill=Wheat!50] (3,-2) rectangle +(16,4);
\draw (0,0) ++(5,0) \Mastermind{SkyBlue,Orange,Yellow,Orchid!70};
\draw (0,5) \Masterresult{black}{black}{black}{black} ++(5,0) \Mastermind{SkyBlue,Orange,Yellow,Orchid!70} +(-1.5,1) node[scale=10] {\rlap{\checkmark}};
\draw (0,9) \Masterresult{gray}{gray}{gray}{gray} ++(5,0) \Mastermind{Orange,SkyBlue,Orchid!70,Yellow};
\draw (0,13) \Masterresult{black}{gray}{gray}{white} ++(5,0) \Mastermind{Orange,Orchid!70,Yellow,red};
\draw (0,17) \Masterresult{black}{white}{white}{white} ++(5,0) \Mastermind{SkyBlue,SkyBlue, red, red};
\draw (0,21) \Masterresult{black}{white}{white}{white} ++(5,0) \Mastermind{Orange,Orange,LimeGreen,LimeGreen};
\end{tikzpicture}
\captionsetup{style=rightside}
\caption{A play of Mastermind}\label{Figure.Mastermind}
\end{wrapfigure}
Wordle can be seen as a variation of an older game known as \emph{Mastermind}, a more difficult and subtle game for the codebreaker, because information is provided only about the number of correct bits in a guessword, rather than specifically about their location.\footnote{Mastermind itself is a variation of the earlier game Cows and Bulls, which is essentially the two-color case.} Mastermind is played between two players, the \emph{codemaker}, who creates a secret codeword sequence of colors chosen from a fixed set $\Sigma$ of available colors, and the \emph{codebreaker}, who strives to discover the codeword by proposing a series of candidate guesswords, receiving feedback about each of them concerning its accuracy. In figure \ref{Figure.Mastermind}, the codemaker has placed the secret codeword in the box at the bottom, hidden from view, and the candidate guesswords of the codebreaker proceed from the top, winning at the fifth guess. For each candidate guessword, the Mastermind feedback is provided with the indicators at left---black indicators count the number of fully correct pegs of the right color in the right position; gray indictors count the number of remaining pegs that could become correct if the incorrect pegs were to be suitably rearranged in the best possible way; and empty indicator holes thus count the number of erroneous pegs, which would remain incorrect after any such rearrangement. As I had mentioned, therefore, the Mastermind indicators give information merely about the number of correct and incorrect pegs, with no direct information about which specific pegs are correct or could become correct with rearrangement, although this further location information can sometimes be inferred, and successfully making these inferences is an important strategic element of the game. The standard finite version of the game has codewords of length four, with six colors available as illustrated by figure \ref{Figure.Mastermind}, although interesting varieties of the game result by varying these parameters.

\enlargethispage{20pt}
Let us consider the natural infinitary versions of Mastermind, where we allow infinitely long codewords and guesswords, and play proceeds through a possibly infinite sequence of guesswords, perhaps transfinitely. In the main case, we shall consider codewords of length $\omega$, with colors taken from a fixed countable set $\Sigma$ of available colors. The game is interesting to consider already with only two colors, but there could be any finite number or infinitely many colors. To play the game, the codemaker selects a hidden codeword $w\in{}^\omega\Sigma$, an infinite sequence of the allowed colors. The codebreaker proposes successive candidate guesswords $s\in{}^\omega\Sigma$, and after each guess receives the Mastermind feedback, three indicators: the correctness count $\kappa$, which is the number of bits (whether finite or infinite) that are fully correct in the guessword; the rearrangement count $\rho$, which is the number of bits that could become correct with a suitable rearrangement of the incorrect bits; and the inherent error count $\varepsilon$, which is the number of bits that must remain incorrect in any such permutation.\footnote{This formulation of the feedback in infinite Mastermind makes what seems to me the most natural choice from several reasonable alternatives, regarding differences that do not arise in the finite case. In particular, the rearrangement permutations here must act only on the incorrect pegs---one is not allowed to move correct pegs around in attempt to free up additional pegs of a color that had occurred infinitely. Notice also that the permutation formulation of the error count is not quite the same as what would arise by allowing injective partial rearrangement functions.} Thus, the Mastermind indicator feedback consists of the three cardinal numbers $(\kappa,\rho,\varepsilon)$, whether finite or countably infinite. For any given codeword and guessword, it is not difficult to see that there is a permutation of the incorrect pegs in the guessword which simultaneously minimizes the number of resulting incorrect pegs while also maximizing the number of correct pegs (minimizing and maximizing the cardinals, not necessarily the sets). The three indicator cardinals $(\kappa,\rho,\varepsilon)$ thus correspond respectively to the information provided in ordinary finite Mastermind by the number of small black indicators, the number of gray indicators, and the number of empty indicator holes. Notice that the cardinal sum $\rho+\varepsilon$ is the total number of incorrect bits in the guessword. A guessword with feedback $(\omega,0,0)$ therefore means that there are infinitely many correct bits and no incorrect bits, and so this is exactly the case where the guessword is identical to the codeword---the codebreaker has won.

In finite Mastermind, there are two standard variations depending on whether the color sequences are allowed to repeat colors or not. In the infinitary context, it will turn out that the no-duplication variation is much harder for the codebreaker to win and exhibits some very subtle connections with set-theoretic independence, which I explore in section \ref{Section.Mastermind-numbers}. But first, let us treat the duplication-allowed variation.

We might naturally allow the game to proceed transfinitely, with the codebreaker making guesses $s_\alpha$ at every ordinal stage $\alpha$. How many stages of guessing will be required? The codebreaker could definitely win with the crude strategy of attempting to guess in turn every single possible codeword. She could enumerate all possible codewords in a well-ordered sequence and simply guess them one after the other until the winning feedback was achieved. But in fact, there is a much quicker method---if there are only countably many colors, then the codebreaker can always win by stage $\omega$, I claim, making the winning guess at stage $\omega$ at the latest.

\begin{theorem}\label{Theorem.Mastermind-win-by-stage-omega}
 In infinitary Mastermind with a countable set of colors and duplication of colors allowed in guesswords, the codebreaker can discover the codeword always by stage $\omega$, and furthermore, she can do this using the same countable sequence of guesswords at all the finite stages, regardless of feedback.
\end{theorem}

\begin{proof}
This argument can be seen as an infinite Mastermind analogue of theorem \ref{Theorem.Wordle-win-by-omega} for infinite Wordle. I shall describe a certain countable collection of guesswords, which the codebreaker can put forth at the finite stages, and the answers to those questions will uniquely determine the codeword, which can then be played at stage $\omega$, if the codebreaker hasn't already hit upon it during play. Specifically, the codebreaker shall inquire systematically about every possible constant-color sequence, as well as every possible nearly-constant color sequence, a sequence that is constant except for one off-color peg. Since there are only countably many colors and countably many positions, she may inquire about all these constant and nearly constant sequences in countably many stages of play.

Let us observe a little about the meaning of the Mastermind response to a constant-color sequence. I claim the response will have the form $(\kappa,0,\varepsilon)$, where $\kappa$ is precisely the number of pegs in the codeword having that color and $\varepsilon$ is the number of pegs in the codeword having some other color. The rearrangement count $\rho$ must be $0$ for a constant-color guessword because there is no way to rearrange the sequence so as to achieve any additional correct pegs, since the sequence remains unchanged by rearranging. In this way, the codebreaker shall come to know the number of pegs of each color in the true codeword.

Similarly, let us consider the feedback indicator response to a nearly-constant guessword, a sequence that is constant except for exactly one off-color bit. Suppose, for example, that $s$ has constant value red except for a blue peg in position $n$, and that we already know that blue occurs at least once in the codeword. If the Mastermind response to this inquiry has the form $(\kappa,0,\varepsilon)$, then I claim the codebreaker may deduce that the codeword must indeed have a blue peg in position $n$, for otherwise that bit would have been incorrect, but we could have made a rearrangement of it to become correct, by placing the blue peg where it belongs,  and so the rearrangement count wouldn't have been $0$. If alternatively, the Mastermind response is $(\kappa,\rho,\varepsilon)$ with a positive rearrangement count $\rho>0$, then I claim that we may deduce that the codeword does not have a blue peg in position $n$, since if it did, that peg would have been included amongst the correct bits and so the permutation of the remaining bits would be rearranging only red pegs, changing nothing, and so $\rho$ couldn't have been nonzero.

Thus, by inspecting the answers to the constant sequences, the codebreaker comes to know which colors occur in the codeword, and by inspecting the answers to all the various nearly constant sequences, the codebreaker comes to know in exactly which position a given color occurs. Thus, in knowing the answers to all the inquires made at the finite stages of play, the codebreaker is in a position to know the codeword with certainty at stage $\omega$, as claimed.
\end{proof}

A modification of the strategy shows that the codebreaker can play so that the guesswords formed at finite stages converge to the true codeword. Namely, whenever the true color of a cell becomes known, then the codebreaker can just place the correct color into that cell henceforth, but otherwise continue with the constant and nearly constant sequences. In this way, the true codeword will be the limit of the guesswords made at the finite stages.

Let me define that a set $S$ of Mastermind guesswords is a \emph{winning} set, if every codeword is uniquely determined by the Mastermind feedback answers to those guesswords. Such a set can be used to win Mastermind at stage $|S|$, and the proof of theorem \ref{Theorem.Mastermind-win-by-stage-omega} shows that indeed there is a countable winning set of guesswords for duplication-allowed infinite Mastermind, allowing the codebreaker to win at stage $\omega$. I shall prove next that this is optimal in the sense that there is no infinite Mastermind analogue to theorem \ref{Theorem.Wordle-finite-alphabet}; that is, although the codebreaker can win infinite Wordle for a finite alphabet at some finite stage of play, there is no such strategy to win infinite Mastermind at a finite stage of play.

\begin{theorem}\label{Theorem.Mastermind-no-finite-strategy}
 In infinite Mastermind with at least two colors, whether or not duplication of colors is allowed, there is no strategy ensuring that the codebreaker can win at some finite stage of play.
\end{theorem}

\begin{proof}
Consider any strategy for the codebreaker, telling her which guesswords to form at each finite stage on the basis of the answers to the earlier guesswords. Since there are only countably many possible Mastermind answers $(\kappa,\rho,\varepsilon)$ to any guessword, it follows that altogether the strategy will produce only countably may guesswords at finite stages. So there must be many codewords that are never guessed at any finite stage, and thus the strategy has not won against them at a finite stage.
\end{proof}

\section{No-duplication Mastermind}

Let us consider a little more carefully the no-duplication variation of infinite Mastermind, in which codewords and guesswords must be sequences of distinct colors from the color set $\Sigma$, with no duplication of colors allowed within any word. I shall use the notation $^\omega\Sigma$ to refer to the set of all $\omega$-sequences over the color set~$\Sigma$ and~$\inj\Sigma$ to refer to the set of injective sequences, that is, sequences of distinct colors, with no duplication of colors.

In the no-duplication case, the standard Mastermind feedback $(\kappa,\rho,\varepsilon)$ for a codeword $w$ on guessword $s$ simplifies a little. The correctness count $\kappa$ is the number of colors placed into the correct position, and the rearrangement count $\rho$ is the number of colors appearing in $s$, but in the wrong position with respect to $w$. For no-duplication color words, the inherent error count $\varepsilon$ will have finite value $n$ if and only if there are exactly $n$ colors of $w$ not appearing in $s$ and exactly $n$ colors of $s$ not appearing in $w$, since a permutation can align these two sets, while making the other colors match. In particular, the inherent error is $0$ if and only if the two sequences use exactly the same colors. If the colors of the guessword are contained within those of the codeword or conversely, then the inherent error count $\varepsilon$ will either be $0$ or $\omega$, depending on whether they use exactly the same colors; even one missing or extra color will cause an error in the permutation that propagates infinitely.

In the no-duplication variation, it turns out that the codebreaker can no longer win with a countable sequence of guesswords.

\begin{theorem}\label{Theorem.Mastermind-no-repetition}
 In no-duplication infinite Mastermind, for no countable ordinal $\gamma$ does the codebreaker have a strategy to win by stage $\gamma$. In particular, there is no countable winning set of guesswords.
\end{theorem}

\begin{proof}
Let me first treat the easier case of an uncountable color space $\Sigma$. Consider any fixed strategy for the codebreaker and any countable ordinal $\gamma$, and let $\<s_\alpha\mid\alpha<\gamma>$ be the resulting play of the game occurring when the Mastermind feedback at every stage is $(0,0,\omega)$, meaning that every guessword $s_\alpha$ is held to have a zero correctness count $\kappa=0$, a zero rearrangement count $\rho=0$, and an infinite inherent error count  $\varepsilon=\omega$. Notice that any codeword sequence consisting entirely of colors not appearing in any $s_\alpha$ would give these answers. Since $\gamma$ is countable and therefore at most countably many colors appear altogether in the guesswords $s_\alpha$, it follows from the uncountability of $\Sigma$ that there are many codewords still available that do not use any of those colors. All such codewords will give those $(0,0,\omega)$ answers for this play, and so the strategy has not succeeded in distinguishing them by stage $\gamma$.\goodbreak

A subtler case occurs, of course, when $\Sigma$ is countably infinite. Again fix any strategy for the codebreaker and any countable ordinal $\gamma$. Let $\<s_\alpha\mid\alpha<\gamma>$ be the play of the game according to this strategy, where the Mastermind feedback given to $s_\alpha$ now has an infinite correctness count $\kappa=\omega$, an infinite rearrangement count $\rho=\omega$, and an inherent error count $\varepsilon$ that is zero or $\omega$, depending on whether $s_\alpha$ uses every color or not, respectively. We now construct distinct codewords $c$ and $d$ that give rise to these answers and are thus not distinguished by the strategy before $\gamma$. The codewords $c$ and $d$ will be suitably generic codewords that use every color of $\Sigma$ exactly once, while having infinitely many points of agreement with each guessword $s_\alpha$ and infinitely many points of disagreement. Since $\gamma$ is a countable ordinal, we may re-enumerate the guesswords in a countable sequence $\<s_{\alpha_n}\mid n<\omega>$. We shall construct $c$ and $d$ by building them up by finite approximations. Namely, at each stage $n$ in this construction process (which is not the same as the game process) we extend the finite approximations so as to ensure additional points of agreement and disagreement with the earlier $s_{\alpha_k}$ for $k<n$. It follows that $c$ and $d$ will ultimately have infinite correctness and rearrangement counts, and the inherent error count will be $0$, if every color appears in $s_\alpha$, and otherwise it will be $\omega$, as discussed above. Thus, the strategy did not distinguish $c$ and $d$ by stage $\gamma$, and so it did not win by that stage.
\end{proof}

The genericity argument would also have worked in the uncountable color space case, if we had given feedback $(\omega,\omega,\omega)$, since generically we could produce many codewords having infinitely many points of agreement with each $s_\alpha$, infinitely many points of disagreement, and infinitely many totally new colors not appearing in any~$s_\alpha$.

\section{Simplified Mastermind}

Let me consider next a variation I call \emph{simplified Mastermind}, where the feedback information on guessword $s$ with codeword $w$ consists of just two cardinals---the correctness and incorrectness counts.
\begin{align*}
  \text{correctness count:}\quad\|s=w\| \quad    & =_{\text{def}}\quad |\set{n\mid s(n)=w(n)}| \\
  \text{incorrectness count:}\quad\|s\neq w\|\quad & =_{\text{def}}\quad |\set{n\mid s(n)\neq w(n)}|
\end{align*}
In simplified Mastermind we thus entirely omit the issue of rearrangements and the number of bits that could become correct by rearrangement. We might also simplify further by considering the correctness-only variation, providing only the correctness count $\|s=w\|$ at each stage, or similarly for the incorrectness-only variation. All of the simplified variations of the game, of course, are harder for the codebreaker to win than unsimplified Mastermind, because less information is provided. Indeed, in simplified Mastermind, it is not possible to discover the codeword by stage $\omega$ or indeed at any countable ordinal stage, as was possible for duplication-allowed infinite Mastermind in theorem \ref{Theorem.Mastermind-win-by-stage-omega}.

\begin{theorem}\label{Theorem.Simplified-Mastermind-no-countable-winning-set}
 In simplified infinite Mastermind with at least two colors, whether or not duplication is allowed, for no countable ordinal $\gamma$ does the codebreaker have a strategy to win by stage $\gamma$. In particular, there is no countable winning set in simplified infinite Mastermind.
\end{theorem}

\begin{proof}
We mount a genericity argument as in the proof of theorem \ref{Theorem.Mastermind-no-repetition}. Suppose that the codebreaker has a strategy directing her how to play at countable ordinal stages. Let us consider the resulting play of the game $\<s_\alpha\mid\alpha<\gamma>$ by following this strategy, where at every stage of play the simplified Mastermind feedback indicates infinitely many points of agreement and infinitely many points of disagreement, that is, giving feedback $(\omega,\omega)$. I shall produce two distinct codewords $c$ and $d$ both giving rise to those answers for that sequence, showing that the strategy cannot have won by stage $\gamma$. Since $\gamma$ is countable, we can enumerate the guesswords in a countable sequence $\<s_{\alpha_n}\mid n<\omega>$. We now build $c$ and $d$ in stages of finite approximations, starting with different colors, and then extending those at each stage $n$ so as to provide additional points of agreement and points of disagreement with all earlier $s_{\alpha_k}$ for $k\leq n$. In the limit, we have thus produced distinct sequences $c$ and $d$ having infinite agreement and infinite disagreement with each $s_\alpha$, as claimed. So the strategy did not separate these two codewords by stage $\gamma$. It follows immediately that there can be no winning countable set of guesswords, since such a set would enable the codebreaker to win at stage $\omega$.
\end{proof}

Theorem \ref{Theorem.Simplified-Mastermind-no-countable-winning-set} shows that there can be no winning strategy for the codebreaker in simplified infinite Mastermind that wins uniformly by some countable ordinal stage. Nevertheless, if the continuum hypothesis holds, then the codebreaker could simply enumerate all possible codewords in order type $\omega_1$ and systematically place them as guesswords; for any given codeword, this strategy will win at the countable stage at which that codeword appears in the enumeration. So the difference is that the codebreaker cannot win uniformly by some fixed countable ordinal stage, even though it is consistent with \ZFC\ that she has a strategy to win (non-uniformly) always at some countable ordinal stage. The \CH\ connection begins to reveal the set-theoretic aspects of infinite Mastermind, which I shall explore more fully in section \ref{Section.Mastermind-numbers}

\section{Madstermind}

Let us consider next an Absurdle-like variation of Mastermind, which I call \emph{Madstermind}, played between the madster and the codebreaker. The madster plays the role of the Mastermind codemaker, giving Mastermind feedback to the guesswords at each stage, while the codebreaker makes her guesses, but there is no requirement for the madster to commit to a particular secret codeword in advance. Rather, like the absurdist in Absurdle, the madster can keep changing his mind about what the codeword might be, although throughout the play at every stage the answers must altogether be consistent with some codeword. Just as in Absurdle, it is as though the madster is thinking of a set of possible codewords, and then providing feedback answers so as to give away as little information as possible. The codebreaker wins at a finite stage of play with a guessword declared entirely correct---the madster should do this only if this is the only word remaining consistent with earlier answers---but the codebreaker can also win after infinite play, if there is no codeword consistent with all the information given during play. The madster aims to survive to $\omega$ with some possible codeword still in his pocket, consistent with all the answers he had given. This is now a two-player game of perfect information.

\begin{theorem}
 In infinite Madstermind with at least two colors and duplication allowed, the madster has a winning strategy---he can play so as to survive through all finite stages, while yet retaining a codeword at $\omega$ consistent with all those answers.
\end{theorem}

\begin{proof}
This argument employs once again the genericity idea used in the proofs of theorems \ref{Theorem.Mastermind-no-repetition} and \ref{Theorem.Simplified-Mastermind-no-countable-winning-set}. Suppose the colors include red and blue. The madster will aim to play always in accordance with the idea that the ultimate codeword will have infinitely many red pegs and infinitely many blue pegs, but no other colors, and furthermore, that those colors will appear in the codeword in a highly generic manner. As play unfolds, the madster will gradually commit in his mind to a finite initial initial segment of this ultimate codeword, but otherwise keep his options open except for the commitment to genericity. Thus, if a guessword has infinitely many red pegs or infinitely blue pegs, then since a generic codeword would find infinitely many places of agreement with those points, he will answer with the correctness count of $\kappa=\omega$, promising to himself to commit to particular instances of those agreement pegs in his later finite approximations. If the guessword has infinitely many red or blue and is not cofinite in either of these, then since generically there will also be many points of disagreement, he will answer with the rearrangement count of $\rho=\omega$, since these incorrect points can be rearranged so as to become correct. He has thus promised to find particular instances of such disagreement in his later finite approximations. If a guessword is cofinitely red or blue, then the madster can extend his finite approximation to cover the finite portion, and then answer with feedback $(\omega,n,\omega)$, where the rearrangement count $\rho=n$ is determined by the finite part of the guessword, since incorrect red or blue colors in that part of the guessword will be able generically to be placed correctly, but no others; the inherent error count $\varepsilon=\omega$ is infinite for a cofinite guessword, since every rearrangement will still be cofinite and hence have infinitely many incorrect bits. Finally, if a guessword uses colors other than red or blue, these will always count toward the inherent error count $\varepsilon$. At each stage, the madster extends the finite commitment so as to fulfill a few more of the promises made at earlier stages, and in this way, he will answer at every stage that is consistent with the ultimate codeword he is gradually revealing. Since no guessword will ever be declared fully correct by this procedure, the madster will survive to $\omega$ and win the play.
\end{proof}

Of course,  in light of theorem \ref{Theorem.Mastermind-win-by-stage-omega}, the madster cannot survive longer than stage $\omega$ for the duplication-allowed version of Madstermind, since the codebreaker can play so as to know the codeword at that stage. In the no-duplication variation, however, then we can generalize theorem \ref{Theorem.Mastermind-no-repetition} to show that the madster can survive through any countable ordinal stage.

\begin{theorem}
 The madster has a winning strategy in no-duplication infinite Madstermind that does not lose at any countable ordinal stage.
\end{theorem}

\begin{proof}
 This is essentially what the proof of theorem \ref{Theorem.Mastermind-no-repetition} shows. If the color space $\Sigma$ is uncountable, he should answer every guessword with $(0,0,\omega)$, and as in the proof of theorem \ref{Theorem.Mastermind-no-repetition} for any countable ordinal $\gamma$ there will be a codeword at stage $\gamma$ that is consistent with all these answers. If alternatively the color space is countable, then he can answer every guessword with $(\omega,\omega,\varepsilon)$, where the inherent error count $\varepsilon$ is either $0$ or $\omega$, depending on whether the guessword uses all the colors or not, and again the argument of theorem \ref{Theorem.Mastermind-no-repetition} shows that for any countable ordinal stage $\gamma$, there will be codewords consistent with those answers.
\end{proof}

\section{The mastermind numbers}\label{Section.Mastermind-numbers}

The \emph{mastermind number} $\mm$, I define, is the size of the smallest winning set for no-duplication infinite Mastermind using words of length $\omega$ over a countably infinite color set.\footnote{The mastermind number $\mm$ should not be confused with the Martin number $\frak{m}(\P)$ for a forcing notion $\P$, defined to be size of the smallest family of dense sets admitting no generic filter, although it will turn out as a consequence of theorem \ref{Theorem.mm=mms=mmeq=d(neq,inj)} that the mastermind number is the same as the Martin axiom number for the forcing to add a Cohen real $\mm=\frak{m}(\C)$.} The duplication-allowed mastermind number $\mm^*$, in contrast, is the size of the smallest winning set for duplication-allowed infinite Mastermind using words of length $\omega$ over a countably infinite color set. In the case of finitely many colors, we denote the corresponding number by $\mm^{*,n}$. These latter numbers, however, are settled by theorem \ref{Theorem.Mastermind-win-by-stage-omega} to have value $\mm^*=\mm^{*,n}=\aleph_0$, whereas in the no-duplication case, theorem \ref{Theorem.Mastermind-no-repetition} shows that $\mm$ is uncountable, and of course it has size at most continuum. $$\mm^*=\mm^{*,n}=\aleph_0<\aleph_1\leq\mm\leq\frak{c}.$$ In this section, I should like to reveal the fundamentally set-theoretic nature of the mastermind number $\mm$ as a cardinal characteristic of the continuum, whose exact value, it turns out, is independent of \ZFC. I take as a goal to relate the mastermind number to the other familiar cardinal characteristics of the continuum, and by the end of the paper, we shall have a complete account of it.

\begin{theorem}\label{Theorem.MA-implies-m=c}
 If Martin's axiom holds, then the mastermind number is the continuum $\mm=\frak{c}$.
\end{theorem}

\begin{proof}
This amounts to a generalization of the genericity ideas used in the proof of theorem \ref{Theorem.Mastermind-no-repetition}. Suppose we have a set $S$ of guesswords with size less than continuum, $|S|<\continuum$. The forcing to add two Cohen reals, in the form of sequences of distinct colors from the fixed countable set of colors, is certainly c.c.c., and there will be $|S|\cdot\omega$ many dense sets that will ensure that those two Cohen reals are distinct, use all the colors, and have the infinite-agreement/infinite-disagreement property with the respect to the guesswords $s\in S$. So by Martin's axiom, there are indeed such generic codewords, and the proof of theorem \ref{Theorem.Mastermind-no-repetition} shows that they will each give rise to the same Mastermind feedback to those guesswords, and so the given set of guesswords does not distinguish all possible codewords. So it is not winning, and thus every winning set must have size continuum.
\end{proof}

What the proof actually shows is that the mastermind number is at least as large as the Martin number for Cohen forcing $\frak{m}(\C)$, the minimal size of a family of dense sets in Cohen forcing for which there is no generic function meeting them all.

\begin{theorem}
 The mastermind number is at least as large as the Martin number $\frak{m}(\C)$ for the forcing $\C$ to add a Cohen real.
 $$\frak{m}(\C)\quad\leq\quad\mm$$
\end{theorem}

\begin{proof}
We view Cohen forcing $\C$ as adding an injective function $g:\omega\to\omega$ by initial segment. The argument of theorem \ref{Theorem.MA-implies-m=c} shows that for any set $S$ of guesswords, if $|S|<\frak{m}(\C)$, then we will find sequences $c$ and $d$ that it does not separate. So the size of the smallest winning set must be at least $\frak{m}(\C)$.
\end{proof}

Let me next define the \emph{simplified mastermind number} $\mms$, which is the size of the smallest winning set in simplified no-duplication infinite Mastermind, using words of length $\omega$ over a countably infinite color set $\Sigma$, with feedback consisting of the correctness and incorrectness counts. We also have the corresponding further simplified mastermind numbers $\mmeq$ and $\mmneq$, for the sizes of the smallest winning sets in correctness-only and incorrectness-only simplified infinite Mastermind, respectively. Precisely because the simplified games are at least as hard for the codebreaker, as they provide less information at each stage, we can easily observe:
$$\mm\quad\leq\quad\mms\quad\leq\quad{\mmeq\atop\mmneq}$$
For the game variation in which duplication of colors in a word is allowed, let me define the corresponding duplication-allowed simplified mastermind numbers $\mms^*$, $\mmeq^*$, and $\mmneq^*$, which are the sizes of the smallest winning sets in the various duplication-allowed simplified Mastermind games. We similarly observe:
$$\mm^*\quad\leq\quad\mms^*\quad\leq\quad{\mmeq^*\atop\mmneq^*}$$
And similarly for the case with finitely many colors, with $\mm^{*,n}$, $\mms^{*,n}$, and so on.

I shall now aim to relate all these numbers to other better-known cardinal characteristics. The \emph{eventually different} number $\frak{d}({\neq^*})$ is the size of the smallest family of functions $S\of{}^\omega\omega$ such that for every $w\in{}^\omega\omega$ there is some $s\in S$ that is eventually different from $w$, meaning that $s$ and $w$ agree on only finitely many bits, $\|s=w\|<\omega$. This is well known to be equal to the covering number $\frak{d}({\neq^*})=\cov(\mathcal{M})$, as proved for example in \cite[theorem 2.4.1]{BartoszynskiJudh1995:Set-theory-on-the-structure-of-the-real-line}. The covering number is also equal to the Martin number $\frak{m}(\C)$ for the forcing $\C$ to add a Cohen real.

\begin{theorem}\label{Theorem.mms*=mmeq*=d(neq)}
The duplication-allowed simplified mastermind numbers, with a countable infinity of colors, are both equal to the eventually different number and therefore also equal to the covering number of the meager ideal and the Martin number for Cohen forcing.
$$\mms^*\quad=\quad\mmeq^*\quad=\quad\frak{d}({\neq^*})\quad=\quad\cov(\mathcal{M})\quad=\quad\frak{m}(\C)$$
\end{theorem}

\begin{proof}
We have already observed above that $\mms^*\leq\mmeq^*$, so let me prove next that $\mmeq^*\leq\frak{d}({\neq^*})$. Suppose that $D\of{}^\omega\omega$ is an eventually different family, a family of functions such that every function is eventually different from an element of $D$, and let $D^*$ be the collection of all finite variations of members of $D$. I claim that $D^*$, which has the same size as $D$, is a correctness-only simplified Mastermind winning set of guesswords. To see this, consider any codeword $w$. Let $s$ be in the family $D$ and eventually different from $w$. The agreement count $\|s=w\|$, therefore, is finite---it is some finite number $k$. By considering the agreement of $w$ with the finite variations of $s$, changing on one coordinate only, we can notice when the agreement changes to $k+1$ or $k-1$ and thereby come to know which color appears at that coordinate in $w$. So $w$ is determined by its correctness counts on the finite variations of $s$, and so $D^*$ is a winning correctness-only simplified Mastermind guessing set of the same size as $D$, which establishes $\mmeq^*\leq\frak{d}({\neq^*})$.

We can complete a cycle of inequalities by proving $\frak{d}({\neq^*})\leq\mms^*$. Suppose that $S$ is a winning duplication-allowed simplified Mastermind guessing set. For each $s\in S$, let $\bar s$ be some other color sequence that has no bits of agreement with $s$, and let $\bar S$ be the set obtained from $S$ by adding all these new sequences. I claim that $\bar S$ is an eventually different set. To see this, suppose that a codeword $w$ is not eventually different from any element of $\bar S$. In particular, it must have infinite agreement with every element of $S$, meaning $\|w=s\|=\omega$ for all $s\in S$ and also infinite agreement with $\bar s$. But since $s$ and $\bar s$ have no agreement at all, $w$ must have infinite disagreement with $s$, meaning $\|w\neq s\|=\omega$. Thus, $w$ has infinite correctness and incorrectness counts with every element of $S$. But in this case, every finite variation of $w$ would also have those infinite counts, and this would contradict the assumption that $S$ was winning for simplified Mastermind, since $S$ would not separate $w$ from its finite variations. So $\bar S$ is an eventually different family, with the same size as $S$, and consequently $\frak{d}({\neq^*})\leq\mms^*$.

The eventually different number $\frak{d}({\neq^*})$, as I had mentioned, is the same as the covering number $\cov(\mathcal{M})$. It is easy to see that $\cov(\mathcal{M})\leq\frak{d}({\neq^*})$, since from any eventually different family, we can construct a covering family of meager sets, since being eventually different from a fixed sequence is a meager property (as having $n$ points of agreement is an open dense property and so having infinitely many points of agreement is co-meager). For the more difficult converse relation $\frak{d}({\neq^*})\leq\cov(\mathcal{M})$, I refer the reader to \cite[theorem 2.4.1]{BartoszynskiJudh1995:Set-theory-on-the-structure-of-the-real-line}.

Finally, the covering number is also equal to the Martin number $\frak{m}(\C)$ for the forcing to add a Cohen real. This is easy to see, since a collection of open dense sets in $\C$, viewed as adding a function $g:\omega\to\omega$, has as complements a collection of closed nowhere dense sets in $^\omega\omega$, which is covering if and only if the original collection admits no generic function. So $\cov(\mathcal{M})\leq\frak{m}(\C)$. Conversely, any covering of $^\omega\omega$ by meager sets can be decomposed into a covering by closed nowhere dense sets, whose complements are a family of open dense sets admitting no generic. So $\cov(\mathcal{M})=\frak{m}(\C)$.
\end{proof}

Essentially the same argument works in the case of finitely many colors, using the eventually different number $\frak{d}({\neq^*},{}^\omega n)$ for the space $^\omega n$ of sequences using $n$ colors. But it turns out that this number is identical to the continuum.

\begin{theorem}\label{Theorem.mms*,n=mmeq*,n=d(neq)}
In duplication-allowed simplified Mastermind with finitely many colors $n\geq 2$, the corresponding simplified mastermind numbers are equal to the eventually different number for $^\omega n$, which is equal to the continuum.
$$\mms^{*,n}\quad=\quad\mmeq^{*,n}\quad=\quad\frak{d}({\neq^*},{}^\omega n)\quad=\quad\continuum$$
\end{theorem}

\begin{proof}
The cycle of inequalities proved in theorem \ref{Theorem.mms*=mmeq*=d(neq)} works almost identically to establish
$\mms^{*,n}\leq\mmeq^{*,n}\leq\frak{d}({\neq^*,{}^\omega n})\leq\mms^{*,n}$, which gives the first two equality claims of this theorem.

What remains is to prove that $\frak{d}({\neq^*},{}^\omega n)=\continuum$. This part of the argument is due to Corey Switzer, who has graciously allowed me to include it. Of course the eventually different number is at most continuum, so for the converse let us suppose we have an eventually different family $D\of{}^\omega n$ of size less than the continuum. We can build a perfect set $P\of{}^\omega n$ for which any distinct $n$ functions $g_0,\ldots,g_{n-1}\in P$ admit infinitely many points $k$ realizing all distinct values $g_i(k)$, that is, all $n$ colors appear at this coordinate $\set{g_i(k)\mid i<n}=n$. This is a simple mutual genericity property, which we can ensure for a perfect set of such functions by growing a tree of approximations to those functions, extending so as to realize the requirements one by one, and letting $P$ be the paths through this tree. Since $P$ is perfect, it has size continuum, and since every $g\in P$ is eventually different from some $f\in D$ and $D$ has size less than the continuum, it follows that some $f\in D$ must be eventually different from at least $n$ (actually uncountably many) distinct $g_0,\ldots,g_{n-1}\in P$. So for large enough $k$, we have $f(k)\neq g_i(k)$ for each $i<n$. But by the mutual genericity property, there will be large values of $k$ for which the values $g_i(k)$ already use up all $n$ colors, and so it is impossible for them all to be different from $f(k)$. This is a contradiction, and so $\frak{d}({\neq^*},{}^\omega n)=\continuum$, as claimed.
\end{proof}

Let us now consider the mastermind number and the simplified no-duplication mastermind numbers, that is, in the case of injective sequences.

\begin{theorem}\label{Theorem.mm=mms=mmeq=d(neq,inj)}
The mastermind number $\mm$, as well as the no-duplication simplified mastermind numbers $\mms$ and $\mmeq$, are all equal to the eventually-different number for injective Baire space $\frak{d}({\neq^*},\inj\omega)$, which is equal to the covering number $\cov(\mathcal{M})$ and the Martin number $\frak{m}(\C)$.
$$\mm\quad=\quad\mms\quad=\quad\mmeq\quad=\quad\frak{d}({\neq^*},\inj\omega)\quad=\quad\cov(\mathcal{M})\quad=\quad\frak{m}(\C)$$
\end{theorem}

\begin{proof}
Once again we have $\mm\leq\mms\leq\mmeq$ simply because the simplified games provide progressively less information to the codebreaker. To prove $\mmeq\leq\frak{d}({\neq^*},\inj\omega)$, we argue in analogy with the proof of theorem \ref{Theorem.mms*=mmeq*=d(neq)}. Namely, if $D\of\inj\omega$ is an eventually different family of injective functions, a family of injective functions such that every injective function is eventually different from one of them, let $D^*$ be the collection of all injective finite variations of a member of $D$. I claim again that $D^*$ is a correctness-only Mastermind winning set of guesswords. Any given injective codeword $w$ is eventually different from some $s\in D$, so the correctness count $\|s=w\|$ has some finite value $k$. What I claim is that $w$ is determined by its correctness counts with the injective finite variations of $s$. This is a little more subtle in the injective case than it was in the duplication-allowed case, since we can't arbitrarily change one color to another, as this might violate injectivity. Nevertheless, let us consider any particular color bit of $s$---suppose it is red in $s$. We can first change that red bit to any particular color not yet appearing in $s$, and if we ever observe a change in the correctness count from $k$ to $k+1$ or to $k-1$, we will know the correct color value of that bit. But perhaps none of those values changed (or perhaps all colors appear already in $s$), and the correct color appears elsewhere in $s$. Let us consider all possible swaps of that red bit with another color bit of $s$, a transposition of two bits. If the red bit were correct in $s$, then after the swap both colors would now be wrong, and so all such swaps would cause a change in the correctness count from $k$ either to $k-1$ or $k-2$, depending on whether the other bit in the swap had been incorrect or correct; but if the red bit were incorrect, then at least one swap would put the right color into that spot (from a previously incorrect spot) or place the wrong color from that spot into its right spot, in both cases causing a change in correctness from $k$ to $k+1$ or $k+2$. If the correctness count went up by $2$, then the new colors are both correct; but if it only went up by $1$, exactly one of them is correct. But we can tell which one by considering all possible further swaps for each of those bits separately. In this way, we see that $w$ is determined by the correctness counts on the various finite modifications of $s$, and so $D^*$ is a winning no-duplication correctness-only simplified Mastermind set. Since $D^*$ has the same size as $D$, we conclude that $\mmeq\leq\frak{d}({\neq^*},\inj\omega)$.

Let me complete a cycle by proving $\frak{d}({\neq^*},\inj\omega)\leq\mm$, which will show that the first four cardinals of the theorem are equal. Suppose that $S$ is a winning no-duplication Mastermind guessing set. For each $s\in S$, let $\bar s$ be a permutation of $s$ that has no bits of agreement with $s$, and let $\bar S$ be the set obtained from $S$ by adding these new sequences. I claim that $\bar S$ is an eventually different set. To see this, suppose that a codeword $w$ is not eventually different from any element of $\bar S$. In particular, it must have infinite correctness count with every element of $S$, meaning $\|w=s\|=\omega$ for all $s\in S$. Since also it will have infinite agreement with $\bar s$, it follows that the rearrangement count with $s$ will be infinite. Furthermore, the inherent error count between $w$ and $s$ depends solely on the sets of colors appearing in these sequences, since it has finite value $k$ if and only if there are exactly $k$ colors appearing in each of them not in the other, but otherwise it will be $\omega$. Each of these Mastermind feedback counts is unchanged by finite variations of $w$ (using the same colors), and so the guessing set $S$ will not distinguish $w$ from these finite variations, contradicting our assumption that it was a no-duplication Mastermind winning set. Therefore, $w$ must have only finite disagreement with some sequence $s\in S$ or $\bar s\in\bar S$. So $w$ is eventually equal to $s$ or $\bar s$, and consequently is eventually different from the other one. So $\bar S$ is an eventually different family of the same size as $S$, and so $\frak{d}({\neq^*},\inj\omega)\leq\mm$.

Finally, we make the connection with the covering number. For this, let us begin by showing $\frak{d}({\neq^*},\inj\omega)\leq\frak{d}({\neq^*})$. Suppose that we have an eventually different family $D\of{}^\omega\omega$, a family of functions such that every function is eventually different from a member of $D$. For each $f\in D$, let $\bar f(n)=\<n,f(n)>$, using a pairing function on $\omega$. Since we can recover $n$ from $\bar f(n)$, the function $\bar f$ is injective. Let $\bar D=\set{\bar f\mid f\in D}$. I claim that this is eventually different for the injective Baire space. Suppose $g:\omega\to\omega$ is injective, and let $g^*(n)=k$, if $g(n)=\<n,k>$, otherwise $0$. So $g^*:\omega\to\omega$. Since $D$ was eventually different, there is some $f\in D$ that is eventually different from $g^*$. Thus, $f(n)\neq g^*(n)$ for all large enough $n$. From this, it follows that $\<n,f(n)>\neq g(n)$ for all large enough $n$, and consequently, $\bar f$ is eventually different from $g$, as desired. Thus we have produced an eventually different family for injective functions with the same size as $D$, and so $\frak{d}({\neq^*},\inj\omega)\leq\frak{d}({\neq^*})$. And we have already discussed earlier the fact that $\frak{d}({\neq^*})=\cov(\mathcal{M})$.

Next, let me observe that the covering number for Baire space is the same as the covering number for the space of injective functions in Baire space, that is, $\cov(\mathcal{M})=\cov(\mathcal{M},\inj\omega)$. For any function $f:\omega\to\omega$, let $\tilde f(n)$ be the $f(n)$th element of the set $\omega\setminus\ran(\tilde f\restrict n)$. That is, whenever we use a value as $\tilde f(k)$ for $k<n$, it is cast out of the set of possibilities for $\tilde f(n)$, and so $\tilde f$ is injective. From $\tilde f$ we can reconstruct $f$ and vice versa, and indeed every injective function arises as $\tilde f$ for some function $f$. So we have a bijection from Baire space to injective Baire space, and it is continuous and the inverse is continuous, because from an initial segment of $f$ we can read off the corresponding initial segment of $\tilde f$ and conversely. So this is a homeomorphism of ${}^\omega\omega$ with $\inj\omega$, and since the covering number is determined by the topology, the two spaces have the same covering number.

Finally, I prove $\cov(\mathcal{M},\inj\omega)\leq\frak{d}({\neq^*},\inj\omega)$ in analogy with the non-injective case, which will show all the cardinals equal, since we already know that $\cov(\mathcal{M})=\frak{m}(\C)$. Suppose that $D\of\inj\omega$ is an eventually different family of injective functions, meaning that every injective function is eventually different from one of them. But the collection of functions that are eventually different from a given function $f$ is a meager set, since it is an open dense property to have at least $n$ places of agreement with $f$, and hence co-meager to have infinitely many points of agreement. Thus, from $D$ we can construct a covering family of meager sets of the same size as $D$, and so $\cov(\mathcal{M},\inj\omega)\leq\frak{d}({\neq^*},\inj\omega)$.
\end{proof}\goodbreak

Let me settle the issue for the incorrectness-only variation, which it turns out is much harder to win, requiring a guessing set of the same size as the space of all color sequences. One might as well just guess every codeword in turn.

\begin{theorem}
The incorrectness-only mastermind numbers, whether or not duplication of colors is allowed and for any countably collection of at least two colors, are all equal to the continuum.
$$\mmneq^{*,n}\quad=\quad\mmneq^*\quad=\quad\mmneq\quad=\quad\frak{c}$$
\end{theorem}

\begin{proof}
In the incorrectness-only variation of the game, the codebreaker receives information at each stage only about the incorrectness count between the guessword $s$ and the codeword $w$
 $$\|s\neq w\|=|\set{n\mid s(n)\neq w(n)}|.$$
Suppose that $S$ is a successful Mastermind guessing set for this game---every codeword is determined by the pattern of its incorrectness counts with the guesswords in $S$ (and the same argument will work here whether or not duplication is allowed). What I claim is that for every possible codeword $w$, there must be some guessword $s\in S$ that has at most finitely many disagreements with $w$. Otherwise, every $s$ would be infinitely different from $w$, and this would be furthermore true of any finite variation of $w$. But this would mean that the family $S$ would not distinguish $w$ from its finite variations, since they would all get the same incorrectness counts $\|s\neq w\|=\omega$. Therefore, $S$ must contain a guessword that differs only finitely from $w$. But since there are continuum many finite-difference classes, it follows that $S$ has size continuum.
\end{proof}

So far, we've considered winning sets of guesswords merely as sets. The codebreaker could use such a set to win the game by playing every guessword in turn, and the results would determine the true codeword. But actually, the codebreaker need not play all the guesswords, since she gets feedback during play, and so I'd like to know how quickly the codebreaker can force a win, using the feedback that is gained during play. For this purpose, let us define the \emph{strategic} mastermind number $\mmstr$ to be the first ordinal stage $\gamma$ for which the codebreaker has a winning strategy ensuring the win occurs before stage $\gamma$. I similarly define for each game variant the corresponding strategic master mind numbers, $\mmstr^*$ for duplication-allowed infinite Mastermind, $\mmstr^{*,n}$ for $n$ colors, $\mmstrs$ for simplified Mastermind, and so on.

It turns out, however, that except for the case of guessing the codeword at stage $\omega$, the strategic mastermind numbers align exactly with the corresponding mastermind numbers.

\begin{theorem}\
\begin{enumerate}
  \item  The duplication-allowed strategic mastermind number is the ordinal $\omega+1$.
  $$\mmstr^*\quad=\quad\mmstr^{*,n}\quad=\quad\omega+1$$
  \item  The strategic mastermind number and strategic simplified numbers are the same as the mastermind number
  $$\mmstr\quad=\quad\mmstrs\quad=\quad\mmstreq\quad=\quad\mm\quad=\quad\cov(\mathcal{M})\quad=\quad\frak{m}(\C)$$
  \item The strategic duplication-allowed simplified mastermind numbers are the same as the simplified mastermind number
  $$\mmstrs^*\quad=\quad\mmstreq^*\quad=\quad\mms^*\quad=\quad\cov(\mathcal{M})\quad=\quad\frak{m}(\C)$$
  \item The corresponding fact is also true in the case of finitely many colors
  $$\mmstrs^{*,n}\quad=\quad\mmstreq^{*,n}\quad=\quad\mms^{*,n}\quad=\quad\continuum$$
 \item And the strategic incorrectness-only mastermind numbers are also the continuum
  $$\mmstrneq^*=\mmstrneq^{*,n}=\continuum$$
\end{enumerate}
\end{theorem}

\begin{proof}
For statement (1), notice that theorems \ref{Theorem.Mastermind-win-by-stage-omega} and \ref{Theorem.Mastermind-no-finite-strategy} show that the codebreaker can win duplication-allowed infinite Mastermind by stage $\omega$, but not earlier. It follows that  $\mmstr^*=\mmstr^{*,n}=\omega+1$.

For statement (2), we easily observe that $\mmstr\leq\mmstrs\leq\mmstreq$ because the simplified games are at least as hard to win for the codebreaker as the unsimplified games. Next, let me argue that $\mmstreq\leq\mmeq$. Fix any winning correctness-only Mastermind set of guesswords, and simply place them one after the other as guesswords in a minimal order type. Eventually, one of the guesswords must achieve a finite correctness count, for otherwise the winning set would not distinguish the codeword from its finite variations. As soon as a finite correctness count is achieved, then we switch to the method of theorem \ref{Theorem.mm=mms=mmeq=d(neq,inj)}, whereby with $\omega$ many additional guesses concering finite variations of that finite-count guessword we will determine the true codeword, and therefore achieve the win at the next simple limit ordinal, which will be before the next cardinal and therefore before $\mmeq$.

Since $\mmeq=\mm$, we complete the cycle by proving that $\mm\leq\mmstr$. For this, it suffices by theorem \ref{Theorem.mm=mms=mmeq=d(neq,inj)} to show that $\frak{m}(\C)\leq\mmstr$. To see this, suppose toward contradiction that the codebreaker has a winning strategy to win always before stage $\gamma$, for some ordinal $\gamma<\frak{m}(\C)$. Consider now the method of theorem \ref{Theorem.Mastermind-no-repetition}, where the madster answers every guessword with feedback $(\omega,\omega,\varepsilon)$, meaning infinite correctness count, infinite rearrangement count, and inherent error count $\varepsilon$ either of $0$ or $\omega$, depending on whether every color appears in the guessword. The madster is thinking of a codeword that is highly generic, meeting every codeword infinitely, disagreeing infinitely, yet using every color exactly once. The point now is that if the game has proceeded for $\gamma$ many steps, then by viewing $\C$ as the forcing to add an injective function $g:\omega\to\omega$ and since $\gamma<\frak{m}(\C)$, we can find such a sequence $g$ that is sufficiently generic with respect to these codewords up to stage $\gamma$. And so the madster retains a successful codeword in his pocket at stage $\gamma$ and so the strategy has not won after all always before stage $\gamma$.

Statement (3) is proved similarly. Namely, $\mmstrs^*\leq\mmstreq^*$ because the game is easier to win with more information. And $\mmstreq^*\leq\mms^*$ because one can use as guesswords an eventually different family, until one is found with a finite correctness count, and then win $\omega$ many stages after that. And finally, we use that $\mms^*=\frak{d}({\neq^*})$ to complete the cycle by proving $\frak{d}({\neq^*})\leq\mmstrs^*$, as follows. The madster can answer infinite agreement and infinite disagreement for all guesswords up to stage $\gamma$, and if $\gamma<\frak{d}({\neq^*})$, then indeed there will be a codeword exhibiting that behavior, and so the codebreaker has not won by that stage.

Statement (4) is proved similarly via the cycle $\mmstrs^{*,n}\leq\mmstreq^{*,n}\leq\mms^{*,n}=\continuum=\frak{d}({\neq^*},{}^\omega n)\leq\mmstrs^{*,n}$.

Finally, statement (5) is proved by observing that for any fewer than continuum many guesswords, there will always be codewords that have infinite disagreement with all of them.
\end{proof}\goodbreak

I summarize the results I have proved about the mastermind numbers in the following diagram:

$$
\begin{tikzpicture}[scale=.9,transform shape,every node/.style={minimum height=5ex,inner sep=8pt,outer sep=0pt,align=center},node distance = 0mm]
\draw node[fill=Yellow!20,text width=3em] (aleph0) {\baselineskip=25pt$\mm^*\break \mm^{*,n}\break \aleph_0$\par}
   node[right = of aleph0.east] (lt) {$<$}
   node[right = of lt.east,fill=Blue!15,text width=3em] (omega+1) {\baselineskip=25pt$\mmstr^*\break \mmstr^{*,n}\break \omega+1$\par}
   node[right = of omega+1.east] (lt) {$<$}
   node[fill=Red!15,right=of lt.east] (aleph1) {$\aleph_1$}
   node[right = of aleph1.east] (leq) {$\leq$}
   node[right = of leq.east, fill=Orchid!20,text width=8em] (cov) {\baselineskip=25pt
     $\mm\qquad\mmstr\break \mms \qquad \mmstrs\break \mmeq\qquad\mmstreq\break
    \mms^* \qquad\mmstrs^*\break \mmeq^*\qquad \mmstreq^*\break
    \frak{d}({\neq^*})\break \frak{d}({\neq^*},\inj\omega) \break
    \cov(\mathcal{M})\break \frak{m}(\C)$\par}
  node[right = of cov.east] (leq) {$\leq$}
  node[right = of leq.east,fill=Orange!20,text width=6.5em] {\baselineskip=25pt
   $\mmneq\qquad \mmstrneq\break \mmneq^*\qquad\mmstrneq^*\break \mmneq^{*,n}\qquad\mmstrneq^{*,n}\break \mms^{*,n}\qquad\mmstrs^{*,n}\break \mmeq^{*,n}\qquad\mmstreq^{*,n}\break
   \frak{d}({\neq^*},{}^\omega n)\break 2^{\aleph_0}\break\continuum$\par};
\end{tikzpicture}
$$

\printbibliography


\end{document}